\documentclass[11pt]{article}

\usepackage[utf8]{inputenc}
\usepackage[T1]{fontenc}
\usepackage{times}
\usepackage{geometry}
\usepackage[english]{babel}
\usepackage{microtype}
\usepackage{pdfpages}
\usepackage{tocloft}
\usepackage{imakeidx}
\usepackage{titling}
\usepackage{csquotes}
\usepackage{amsmath, amssymb, amsthm}
\usepackage{todonotes, verbatim}
\usepackage{faktor}
\usepackage{fancyhdr}
\usepackage{pgfplots}
\usepackage{adjustbox}
\usepackage{tikz-cd}
\usetikzlibrary{decorations.pathmorphing}
\usepackage{multicol}

\usepackage{hyperref}
\usepackage{cleveref}

\graphicspath{{./images/}}
\pgfplotsset{compat=1.15}

\hypersetup{
colorlinks=true,
citecolor=[rgb]{0,0,0.8},
linkcolor=[rgb]{0,0,0.8},
}

\theoremstyle{thmstyleone}
\newtheorem{theorem}{Theorem}
\newtheorem{corollary}{Corollary}
\newtheorem{lemma}{Lemma}

\theoremstyle{thmstyletwo}

\newtheorem{question}{Question}
\newtheorem{acknowledgements}{Acknowledgements}

\theoremstyle{thmstylethree}
\newtheorem{definition}{Definition}

\newcommand{\Z}{\ensuremath{\mathbb{Z}}}

\tikzset{
node distance=2.3cm, 
every state/.style={thick, fill=gray!10},  
auto,
snake it/.style={decorate, decoration=snake}
}
\tikzstyle{every path}=[line width=0.8pt]

\title{Formality of spaces with Lusternik-Schnirelmann category 1}
\author{Torgeir Aamb\o}
\date{}

\begin{document}
\maketitle

\begin{abstract}
    It is a well known fact that formal dg-algebras admit no non-trivial Massey products, 
    while the converse fails. We prove that by restricting to dg-algebras whose induced 
    product on cohomology is trivial, we do in fact get this converse. This allows us to 
    prove that spaces of Lusternik-Schnirelmann category $1$ are formal spaces.
\end{abstract}

\section{Introduction}

The notion of formal dg-algebras, being algebras that are quasi-isomorphic to their cohomology 
algebra, was introduced in \cite{DGMS} to solve problems in rational homotopy theory. In 
this paper, the authors remark that having no non-vanishing Massey 
$n$-products is a weaker property than being formal, meaning that Massey $n$-products 
serve as obstructions to formality. They claim that formality is equivalent to a 
``uniform vanishing'' – a stronger version of just being vanishing. In later times 
the study of dg-algebras have been explored further using the theory of 
$A_\infty$-algebras. Any dg-algebra $A$ can be viewed as a ``trivial'' 
$A_\infty$-algebra, which ables us to think about general $A_\infty$-algebras as homotopy 
theoretic versions of dg-algebras. In \cite{kadeishvili} Kadeishvili proved that the 
cohomology algebra $H(A)$ of a dg-algebra $A$ naturally admits an $A_\infty$-structure in 
such a way that there is a quasi-isomorphism of $A_\infty$-algebras $H(A)\longrightarrow A$. 
The higher products on this $A_\infty$-structure are often claimed to be the Massey 
products, but this is not always true as the $n$-ary product of the $A_\infty$-structure on 
$H(A)$ is not always a representative of the Massey $n$-product \cite{detection}. It is 
however the case that the vanishing of these higher arity products on $H(A)$ is a stronger 
condition than the vanishing of the Massey $n$-products \cite{AHO} – in fact, if all the 
higher products vanish, then the dg-algebra is formal. Hence having vanishing 
$A_\infty$-structure on $H(A)$ is equivalent to $A$ being formal. Using the equivalent 
definition of formality from \cite{keller} this is true almost by definition. 

In \cite{detection} the authors prove that even though the higher products on $H(A)$ might 
not be the Massey products, they are so up to a sum of lower degree products. We use 
this to prove that vanishing Massey products is equivalent to formality in the case 
that the cohomology algebra has trivial products. Hence the vanishing Massey products 
on spaces with vanishing cup products are always uniformly vanishing in the sense of 
\cite{DGMS}. 

There is a property of a space that allows us to know an upper bound for its 
cup-length, namely the Lusternik-Schnirelmann category of the space. This is an 
integer describing how a space can be glued together by cones. If we limit our study 
to spaces with Lusternik-Schnirelmann $1$, we know that the cup-length is $0$, 
meaning that we have vanishing products on reduced cohomology. Using the above result
we then prove that spaces with Lusternik-Schnirelmann category $1$ are formal. This 
generalizes the formality for suspended spaces, proven in \cite{FHT}.

The result is most likely already known by specialists, at least in the rational case.
But, this seems to be a new method of proving it. Alternatively one can use that any 
space $X$ with $\text{cat}_{LS}(X)=1$ is a co-H-space \cite{hess}, and then that any 
co-H-space is a wedge of spheres \cite{co-H-space}. Formality is preserved under the 
wedge product \cite{hess}, and since spheres are formal, we know that any co-H-space, 
and thus any space $X$ with $\text{cat}_{LS}(X)=1$ is a formal space. 

\begin{acknowledgements}
This work is the product of the authors master thesis at NTNU
supervised by Gereon Quick. We want to thank him for an interesting project, and for 
all the helpful feedback and comments. We also thank José Manuel Moreno-Fernández for 
providing comments on an earlier version of this paper. 
\end{acknowledgements}

\section{\texorpdfstring{$A_\infty$}{A}-algebras}

We provide a quick overview of the most important theory on $A_\infty$-algebras. 
For a more in-depth treatment, see \cite{keller}. 
Let $K$ be a field of characteristic $0$. 

\begin{definition}
    An $A_\infty$-algebra $(A, m)$, over $K$ is a $\Z$-graded vector 
    space $A=\bigoplus_{i\in \Z}A_i$ together with a family of $K$-linear maps 
    $m_n : A^{\otimes n}\longrightarrow A$ of degree $2-n$, such that the identities
    $$\sum_{r+s+t = n}(-1)^{r+st}m_{r+1+t} (Id^{\otimes r}\otimes m_s \otimes Id^{\otimes t}) = 0$$
    hold for all $n, s\geq 1$ and $r, t\geq 0$.
\end{definition}

These relations are called the coherence relations, or the Stasheff identities in $A$. 
For $n=1$ the coherence relation simply becomes $$0 = (-1)^{0+0}m_1 (m_1) = m_1^2 .$$ This 
means that $A$ is a cochain complex with differential $d=m_1$ as $m_1$ is a degree 
$1$ map. For $n=2$ we get
$$0 = (-1)^{1}m_2(Id\otimes m_1)+(-1)^{0}m_1 (m_2)+(-1)^{1}m_2 (m_1\otimes Id)$$
which reduces to $m_1 m_2 = m_2(m_1\otimes Id + Id\otimes m_1)$. This means that $m_1$ is 
a derivation with respect to $m_2$ as $m_2$ has degree $2$, usually stated as satisfying 
the Leibniz rule. The standard Leibniz rule comes out of this formula when applying it to 
elements and using the Koszul grading rule.

The third relation tells us that $m_2$ is not necessarily associative, and that the 
associator is given by $m_3$. This is usually referred to as $m_2$ being associative up 
to homotopy, which gives $A_\infty$-algebras their other name, strong homotopy associative 
algebras or sha-algebras for short. If $m_3=0$ (or $m_1=0$) this relation reduces to the 
associator being zero, which means that $m_2$ is an associative product. 

\begin{definition}
    A dg-algebra is an $A_\infty$-algebra $(A, m)$ where $m_i = 0$ for 
    $i \geq 3$. For simplicity we usually just denote it by $A$. 
\end{definition}

By the equations above describing the Stasheff identities, this definition is equivalent 
to the classical definition, i.e. a $\Z$-graded vector space with an associative product 
and a differential satisfying the Leibniz rule. 

\begin{definition}
    Let $(A, m^A)$ and $(B, m^B)$ be $A_\infty$-algebras. A morphism of 
    $A_\infty$-algebras $f:A\longrightarrow B$, also called $A_\infty$-morphism, is a 
    family of linear maps $f_n:A^{\otimes n}\longrightarrow B$ of degree $1-n$, such that
    $$\sum_{n = r+s+t}(-1)^{r+st}f_{r+1+t}(id^{\otimes r}\otimes m_s^A \otimes id^{\otimes t}) = \sum_{k=1}^{n}\sum_{n=i_1+\cdots i_k}(-1)^{u_k} m_k^B(f_{i_1}\otimes \cdots \otimes f_{i_k})$$
    where $u_k=\displaystyle \sum_{t=1}^{k-1}t(i_{k-t}-1)$.     
\end{definition}

We call $f$ an $A_\infty$-isomorphism, or an isomorphism of $A_\infty$-algebras, if $f_1$ 
is an isomorphism of chain complexes.

Since any $A_\infty$-algebra $(A, m)$ has a map $m_1$ such that $m_1^2=0$, we can also 
create its cohomology algebra, denoted $H(A)$. The cohomology algebra of a dg-algebra 
is a graded associative algebra with the induced product from $A$, which we can treat as 
a dg-algebra by letting it have trivial differential.

Let $f:A\longrightarrow B$ be an $A_\infty$-morphism. We call $f$ an 
$A_\infty$-quasi-isomorphism, or a quasi-isomorphism of $A_\infty$-algebras, if $f_1$ is 
a quasi-isomorphism of chain complexes, i.e. it induces an isomorphism on their cohomology 
algebras.

Notice that if $f_j=0$ for $j\geq 2$ and $m^A_i = 0 = m^B_i$ for $i\geq 3$, i.e. $A$ and 
$B$ are dg-algebras, then this definition reduces to the standard quasi-isomorphisms of 
dg-algebras.

\begin{definition}
    Let $A$ be a dg-algebra. We say $A$ is formal if it is 
    $A_\infty$-quasi-isomorphic to a dg-algebra with trivial differential.        
\end{definition}

A formal dg-algebra is isomorphic to $H(A)$, so we can define a dg-algebra to be 
formal if it is $A_\infty$-quasi-isomorphic to its cohomology algebra. We also remark 
that this is not the most classical definition, which uses a zig-zag of 
dg-quasi-isomorphisms instead of an $A_\infty$-quasi-isomorphism. This is equivalent 
to the one we are using here. See \cite{AHO} and \cite{keller} for further details. 

Formal dg-algebras are very nice algebras that as mentioned have their historical 
upbringing in rational homotopy theory. Examples include the Sullivan algebras of 
Kähler manifolds \cite{DGMS}.

\section{Massey products}

Massey products are partially defined higher order cohomology operations on dg-algebras
introduced by Massey in \cite{Massey}. Since they are only partially defined we need an 
easy way to package this information in. This is done through defining systems. Let $A$ 
be a dg-algebra and denote $\bar{x} = (-1)^{|x|}x$. 

\begin{definition}
    A defining system for a set of cohomology classes $x_1, \ldots, x_n$ 
    in $H(A)$ is a collection $\{ a_{i,j}\}$ of cochains in $A$ such that
    \begin{enumerate}
        \item $[a_{i-1, i}] = x_i$
        \item $d(a_{i, j}) = \sum_{i<k<j}\overline{a_{i, k}}a_{k, j}$
    \end{enumerate}
    for all pairs $(i,j)\neq (0,n)$ where $i\leq j$.
\end{definition}

\begin{definition}
    The Massey $n$-product of $n$ cohomology classes $x_1, \ldots, x_n$, 
    denoted $\langle x_1, \ldots, x_n\rangle$ is defined to be the set of all $[a_{0,n}]$, 
    where $$a_{0,n} = \sum_{0<k<n}\overline{a_{0, k}}a_{k, n}$$ such that $\{ a_{i,j} \}$ 
    is a defining system.         
\end{definition}

For $n=2$ this is just the induced product on cohomology, up to a sign. For $n=3$ this 
is the classical triple Massey product. When we use the phrase ``all Massey $n$-products'', 
we mean all Massey $n$-products for $n\geq 3$. 

The fact that multiple cohomology classes can be in this set means that the product is 
only partially defined. If this set contains just a single class, then we say the Massey 
product is uniquely defined. What matters for us is when these products are ``trivial''. 

\begin{definition}
    We say that the Massey $n$-product vanishes if it contains zero as an 
    element, i.e. $0\in \langle x_1, \ldots, x_n\rangle$.    
\end{definition}

\section{Results on formality}

One of the reasons we are interested in Massey products is that they serve as an 
obstruction to formality. Intuitively, if Massey products exist on the algebra of 
cochains $C(X)$ on a topological space $X$, this means that $C(X)$ contains more 
information about our space than the cohomology ring of $X$. If a non-trivial Massey 
product exists on $C(X)$, this means that there will always exist another space $Y$, 
not homeomorphic to $X$, but with the same cohomology ring as $X$. The most famous 
example of this are the Borromean rings. These are a set of three rings, every pair of 
them not linked with the two others, but all three still linked. This has the same 
cohomology ring as three unlinked circles, but the Borromean rings admit a non-vanishing 
Massey product, while the three unlinked circles does not. Hence the algebra of forms 
on the Borromean rings can not be formal. This means that Massey products detect 
non-formality. This is summarized into the following theorem.

\begin{theorem}[\cite{DGMS}]
    Let $(A, m)$ be a formal dg-algebra. Then all Massey $n$-products vanish.         
\end{theorem}

Unfortunately, knowing a dg-algebra has no non-vanishing Massey products is not enough 
to determine that it is formal. But using the full $A_\infty$-algebras we can get closer 
to some version of this being true.

\begin{theorem}[\cite{kadeishvili}]
    \label{thm:Kadeishvili}
    Let $(A,m)$ be a dg-algebra. Then there exists an (up to $A_\infty$-isomorphism) 
    unique $A_\infty$-algebra structure on its cohomology algebra $H(A)$ with $m_1=0$, 
    $m_2$ the induced product from $A$, and a quasi-isomorphism of $A_\infty$-algebras 
    $H(A)\longrightarrow A$.
\end{theorem}

Note that this does not mean that all dg-algebras are formal, as now $H(A)$ is not 
necessarily just a dg-algebra anymore. We can think of this higher structure on $H(A)$ 
as measuring how far away $A$ is from being formal. Since $m_1=0$ we get that the 
product $m_2$ is associative, but not for the reason we mentioned earlier. This means 
that these higher products no longer are interpreted as homotopies, but instead as 
something more like Massey products. Hence we call them the ``higher products'' on $H(A)$.

We said that the $A_\infty$-structure measures how far away $A$ is from being formal. 
We noted earlier that an $A_\infty$-algebra with $m_k=0$ for $k\geq 3$ is a dg-algebra. 
This means that if the $A_\infty$-structure on $H(A)$ has $m_k=0$ for $k\geq 3$, then 
$A$ is formal, as we then have an $A_\infty$-quasi-isomorphism between two dg-algebras. 
This is rather important so we state it as a theorem.

\begin{theorem}[\cite{AHO}]
    \label{thm:AHO_formal}
    Let $(A, m)$ be a dg-algebra. Then $A$ is formal if and only if all the higher products
    on $H(A)$ vanish.        
\end{theorem}

One direction of the proof is by definition, as described above. The other part is 
because the $A_\infty$-structure on $H(A)$ is unique up to isomorphism of 
$A_\infty$-algebras.

In \cite{DGMS} the authors say that formality is equivalent to a uniform vanishing of the 
Massey products. We are intuitively able to choose the zero element as our Massey 
products in such a nice uniform way that they are a part of an $A_\infty$-structure, 
which must mean that we have formality, as above.

We want to use this to get an idea of ``how close'' the normal Massey products are 
from being sufficient obstructions. What we mean by this is that we want to find a case 
where vanishing Massey products mean we have a formal dg-algebra. To get to such a result 
we first need a theorem that connects the higher products to normal Massey products.

\begin{theorem}[\cite{detection}]
    Let $A$ be a dg-algebra and $x\in \langle x_1, \ldots, x_n\rangle$ with $n\geq 3$. 
    Then for any $A_\infty$-structure on $H(A)$ we have 
    $$\epsilon m_n(x_1, \ldots, x_n) = x+\Gamma$$ where 
    $\Gamma \in \sum_{j=1}^{n-1}\text{Im}(m_j)$ and 
    $\epsilon = (-1)^{\sum_{j=1}^{n-1} (n-j)|x_j|}$.        
\end{theorem}

\begin{corollary}[\cite{detection}]
    \label{cor:detection_unique}
    Let $A$ be a dg-algebra and $m$ an $A_\infty$-structure on its cohomology $H(A)$ such 
    that $m_k = 0$ for all $1 \leq k \leq n-1$. Then the Massey $n$-product 
    $\langle x_1, \ldots, x_n \rangle$  is uniquely defined for any set of cohomology 
    classes $x_1, \ldots, x_n$. Furthermore the unique element in the Massey product is 
    recovered by $m_n$ up to a sign, i.e. 
    $\langle x_1, \ldots, x_n \rangle = \epsilon m_n(x_1, \ldots, x_n)$, where again 
    $\epsilon = (-1)^{\sum_{j=1}^{n-1} (n-j)|x_j|}$.
\end{corollary}

By this we get our first result.

\begin{theorem}
    \label{thm:1}
    Let $A$ be a dg-algebra and $H(A)$ its cohomology algebra. If the induced product on 
    $H(A)$ is trivial and all Massey $n$-products on $A$ vanish, then $A$ is formal.        
\end{theorem}

\begin{proof}
    By \cref{thm:Kadeishvili} we know that $H(A)$ can be equipped with the structure 
    $\{m_i\}$ of an $A_\infty$-algebra such that $m_1=0$ and $m_2$ is the product induced 
    from $A$, which is assumed to be trivial. We claim that $m_k = 0$ for all $k\geq 3$ 
    as well, and hence that $A$ is formal by \cref{thm:AHO_formal}. We prove this 
    claim by induction. 
    
    Since $m_2=0$ we know that all Massey triple products are defined. By the below 
    induction argument all the higher Massey products will be defined as well. Since 
    $m_2=0$ we already have our base case.
    
    Assume $m_k = 0$ for $1\leq k\leq n-1$. By \cref{cor:detection_unique} we know 
    that $\langle x_1, \ldots, x_n \rangle$ consists of a unique element for all choices of 
    classes $x_1, \ldots, x_n$. This element is by assumption the zero class, as we assumed 
    all Massey products to be vanishing. This class is recovered up to a sign by $m_n$, 
    which means $m_n(x_1,\ldots, x_n)=0$ for all choices of $x_1, \ldots, x_n$. Hence 
    $m_n=0$ and we are done. 
\end{proof}

This proof shows that when we have a trivial induced product on $H(A)$, the vanishing 
Massey products neatly forms a trivial $A_\infty$-structure on $H(A)$, which we earlier 
said was the way of interpreting the uniform vanishing.

It is tempting to think that having trivial product in cohomology also makes every attempt 
to build and produce a Massey product impossible. But there are examples of this not being 
the case. One example is the free loop space of an even-dimensional sphere. Its cohomology 
algebra has trivial product, and it is shown in \cite{nonformal_loop} to have non-zero 
Massey products. Hence it can not be formal. We also mentioned the Borromean rings earlier, 
which gives another example. 

\begin{question}
Is there a more general procedure for choosing elements in all the Massey 
products in such a way that they form an $A_\infty$-structure?
\end{question}

\section{The Lusternik-Schnirelmann category}

For topological spaces the criteria of having trivial cup products is quite strong. Say 
we have a path-connected topological space $X$. It has zeroth cohomology $H^0(X;K)\cong K$. 
Hence the requirement to have trivial cup product reduces to requiring $H^i(X;K)=0$ for 
$i>0$, which is really limiting. One solution to this is to look at reduced cohomology. 

\begin{definition}
    Let $X$ be a topological space and $C^\ast(X;K)$ its cochain complex (treated here as a 
    dg-algebra). We define its augmented cochain dg-algebra, denoted 
    $\widetilde{C}^\ast(X;K)$, by adding a copy of the ground field $K$ injectively farthest 
    to the left, i.e.
    \begin{equation*}
        \cdots \longrightarrow 0 \longrightarrow K \overset{\epsilon}\longrightarrow 
        C^0(X;K)\longrightarrow C^1(X;K) \longrightarrow \cdots \longrightarrow C^n(X;K) 
        \longrightarrow \cdots  
    \end{equation*}
\end{definition} 

The cohomology algebra of the augmented cochain complex is called the reduced cohomology 
algebra of $X$ and is denoted $\widetilde{H}^\ast(X;K)$. If the space $X$ is connected, 
then $\widetilde{H}^0(X;K)=0$, meaning that we have completely removed the problem 
described above. 

Spaces with trivial cup product are not in abundance -- even in reduced cohomology -- 
but examples include the spheres and more generally a suspended space. There are ways to 
make sure that a space has a trivial cup product and one of these is using the 
Lusternik-Schnirelmann category. 

\begin{definition}
    Let $X$ be a topological space. The Lusternik-Schnirelmann category of $X$, denoted 
    $\text{cat}_{LS}(X)$, is the least integer $n$ such that there is a covering of $X$ 
    by $n+1$ open subsets $U_i$ that are all contractible in $X$, i.e. their inclusion 
    into $X$ is null-homotopic.    
\end{definition}

This invariant was originally developed in \cite{lscat} as an invariant on manifolds to be 
a lower bound for the number of critical points any real valued function on it could have. 
It has since become a useful – but very difficult to calculate – invariant of topological 
spaces. 

Recall that the cup-length of a topological space 
$X$, denoted $\text{cl}(X)$ is the largest integer $n$ such that a chain 
$[x_1]\cup \cdots \cup [x_n]$ of cohomology classes with $\deg|x_i|\geq 1$ is non-zero. 
We have the following fundamental relation between the cup length and the 
Lusternik-Schnirelmann category of $X$.

\begin{lemma}[\cite{lscategorybook}]
    Let $X$ be a topological space. Then the cup length of $X$ is a lower bound for its 
    Lusternik-Schnirelmann category, i.e. $\text{cl}(X)\leq \text{cat}_{LS}(X)$.    
\end{lemma}

Thus, choosing spaces with $\text{cat}_{LS}(X) = 1$ means we have trivial product on 
reduced cohomology. Examples of such spaces are again the suspended spaces, as they are 
the union of two cones. The last thing we need to know is whether spaces with 
Lusternik-Schnirelmann category $1$ admit any non-vanishing Massey  $n$-products. Luckily, 
by Rudyak we know that they dont.

\begin{theorem}[\cite{Rudyak}]
    \label{thm:rudyak}
    Let $X$ be a topological space with $\text{cat}_{LS}(X)\leq 1$ and 
    $x_1, \ldots, x_n \in \widetilde{H}^\ast(X;K)$. If the Massey $n$-product 
    $\langle x_1, \ldots, x_n\rangle$ is defined then 
    $\langle x_1, \ldots, x_n\rangle = 0$.    
\end{theorem}

We call spaces with formal reduced cochain algebras reduced formal spaces. This means now
that spaces with Lusternik-Schnirelmann category $1$ are reduced formal by 
\cref{thm:rudyak} and \cref{thm:1}. 

\begin{corollary}
    \label{cor:reduced_formal}
    Let $X$ be a topological space with $\text{cat}_{LS}(X)\leq 1$. Then 
    $\widetilde{C}^\ast(X;K)$ is a formal dg-algebra. 
\end{corollary}

We still have to deal with the degree $0$ cochains in order to say that a space $X$ with
$\text{cat}_{LS}(X)\leq 1$ is formal and not just reduced formal. To do this we must 
construct the neccesary span of dg-quasi-isomorphism.

\begin{theorem}
    \label{thm:reduced_formal_formal}
    Any reduced formal topological space $X$ is formal.     
\end{theorem}

\begin{proof}
    Since $X$ is reduced formal we know that there is a span of dg-quasi-isomorphisms 
    $$\widetilde{H}^\ast(X;K)\longleftarrow B \longrightarrow \widetilde{C}^\ast(X;K)$$ for some 
    dg-algebra $B$:

    \begin{center}
        \begin{tikzpicture}
            \node (1) {$k$};
            \node (2) [node distance=3cm, right of=1] {$C^0(X;K)$};
            \node (3) [node distance=3cm, right of=2] {$C^1(X;K)$};
            \node (4) [node distance=3cm, right of=3] {$C^2(X;K)$};
            
            \node (5) [node distance=1.6cm, below of=1] {$B^{-1}$};
            \node (6) [node distance=3cm, right of=5] {$B^0$};
            \node (7) [node distance=3cm, right of=6] {$B^1$};
            \node (8) [node distance=3cm, right of=7] {$B^2$};
            
            \node (9) [node distance=1.6cm, below of=5] {$0$};
            \node (10) [node distance=3cm, right of=9] {$0$};
            \node (11) [node distance=3cm, right of=10] {$H^1(X;K)$};
            \node (12) [node distance=3cm, right of=11] {$H^2(X;K)$};
            
            \node (13) [node distance=2cm, left of=1] {$\cdots$};
            \node (14) [node distance=2cm, left of=5] {$\cdots$};
            \node (15) [node distance=2cm, left of=9] {$\cdots$};
            \node (16) [node distance=2cm, right of=4] {$\cdots$};
            \node (17) [node distance=2cm, right of=8] {$\cdots$};
            \node (18) [node distance=2cm, right of=12] {$\cdots$};
            
            \draw [-to] (13) to node {} (1);
            \draw [-to] (14) to node {} (5);
            \draw [-to] (15) to node {} (9);
            \draw [-to] (4) to node {} (16);
            \draw [-to] (8) to node {} (17);
            \draw [-to] (12) to node {} (18);
            
            \draw [-to] (1) to node {$\epsilon$} (2);
            \draw [-to] (2) to node {$d^0$} (3);
            \draw [-to] (3) to node {$d^1$} (4);
            
            \draw [-to] (5) to node {$d^{-1}_B$} (6);
            \draw [-to] (6) to node {$d^0_B$} (7);
            \draw [-to] (7) to node {$d^1_B$} (8);
            
            \draw [-to] (9) to node {} (10);
            \draw [-to] (10) to node {} (11);
            \draw [-to] (11) to node {$0$} (12);
            
            \draw [-to] (5) to node {$q^{-1}$} (1);
            \draw [-to] (5) to node {} (9);
            
            \draw [-to] (6) to node {$q^0$} (2);
            \draw [-to] (6) to node [swap]{$p^0$} (10);
            
            \draw [-to] (7) to node {$q^1$} (3);
            \draw [-to] (7) to node [swap]{$p^1$} (11);
            
            \draw [-to] (8) to node {$q^2$} (4);
            \draw [-to] (8) to node [swap]{$p^2$} (12);
        \end{tikzpicture}
        \end{center}

    By removing the copy of $k$ from the augmented cochain complex, we can insert
    a copy of $k$ as $H^0(X;K)$ and form a new span of dg-quasi-isomorphisms:

    \begin{center}
        \begin{tikzpicture}
            \node (1) {$0$};
            \node (2) [node distance=3cm, right of=1] {$C^0(X;K)$};
            \node (3) [node distance=3cm, right of=2] {$C^1(X;K)$};
            \node (4) [node distance=3cm, right of=3] {$C^2(X;K)$};
            
            \node (5) [node distance=1.6cm, below of=1] {$0$};
            \node (6) [node distance=3cm, right of=5] {$B^0\oplus B^{-1}$};
            \node (7) [node distance=3cm, right of=6] {$B^1$};
            \node (8) [node distance=3cm, right of=7] {$B^2$};
            
            \node (9) [node distance=1.6cm, below of=5] {$0$};
            \node (10) [node distance=3cm, right of=9] {$k$};
            \node (11) [node distance=3cm, right of=10] {$H^1(X;K)$};
            \node (12) [node distance=3cm, right of=11] {$H^2(X;K)$};
            
            \node (13) [node distance=2cm, left of=1] {$\cdots$};
            \node (14) [node distance=2cm, left of=5] {$\cdots$};
            \node (15) [node distance=2cm, left of=9] {$\cdots$};
            \node (16) [node distance=2cm, right of=4] {$\cdots$};
            \node (17) [node distance=2cm, right of=8] {$\cdots$};
            \node (18) [node distance=2cm, right of=12] {$\cdots$};
            
            \draw [-to] (13) to node {} (1);
            \draw [-to] (14) to node {} (5);
            \draw [-to] (15) to node {} (9);
            \draw [-to] (4) to node {} (16);
            \draw [-to] (8) to node {} (17);
            \draw [-to] (12) to node {} (18);
            
            \draw [-to] (1) to node {} (2);
            \draw [-to] (2) to node {$d^0$} (3);
            \draw [-to] (3) to node {$d^1$} (4);
            
            \draw [-to] (5) to node {} (6);
            \draw [-to] (6) to node {$[d^0_B, 0]$} (7);
            \draw [-to] (7) to node {$d^1_B$} (8);
            
            \draw [-to] (9) to node {} (10);
            \draw [-to] (10) to node {$0$} (11);
            \draw [-to] (11) to node {$0$} (12);
            
            \draw [-to] (5) to node {} (1);
            \draw [-to] (5) to node {} (9);
            
            \draw [-to] (6) to node {$[q^0, 0]$} (2);
            \draw [-to] (6) to node [swap]{$[0, q^{-1}]$} (10);
            
            \draw [-to] (7) to node {$q^1$} (3);
            \draw [-to] (7) to node [swap]{$p^1$} (11);
            
            \draw [-to] (8) to node {$q^2$} (4);
            \draw [-to] (8) to node [swap]{$p^2$} (12);
        \end{tikzpicture}
        \end{center}

    The squares in the diagram commutes due to the original squares commuting
    and the cohomologies also agree by construction. Thus we have a span of 
    dg-quasi-isomorphisms $$H^\ast(X;K)\longleftarrow B'\longrightarrow C^\ast(X;K),$$ 
    which means $X$ is formal. 
\end{proof}

We are now ready to conclude with our main result.
\begin{theorem}
    Let $X$ be a space with $\text{cat}_{LS}(X)\leq 1$. Then $X$ is formal.     
\end{theorem}

\begin{question}
    Are there any formal spaces with trivial cup product and Lusternik-Schnirelmann 
    category greater than $1$? 
\end{question}

\bibliographystyle{alpha}
\bibliography{references}

\begin{thebibliography}{BMFM20}

\bibitem[Bas15]{nonformal_loop}
Somnath Basu.
\newblock Of sullivan models, massey products, and twisted pontrjagin products.
\newblock {\em Journal of homotopy and related structures}, 10:239--273, 2015.

\bibitem[BMFM20]{detection}
Urtzi Buijs, José~Manuel Moreno-Fernández, and Aniceto Murillo.
\newblock A-infinity structures and massey products.
\newblock {\em Mediterranean Journal of Mathematics}, 17, 2020.

\bibitem[CLOT03]{lscategorybook}
Octav Cornea, Gregory Lupton, John Oprea, and Daniel Tanré.
\newblock {\em Lusternik-Schnirelmann category}, volume 103.
\newblock Mathematical Surveys and Monographs, 2003.

\bibitem[DGMS75]{DGMS}
Pierre Deligne, Phillip Griffiths, John Morgan, and Dennis Sullivan.
\newblock Real homotopy theory of kähler manifolds.
\newblock {\em Inventiones math}, 29:245--274, 1975.

\bibitem[FHT01]{FHT}
Yves Félix, Stephen Halperin, and Jean-Claude Thomas.
\newblock {\em Rational Homotopy Theory}.
\newblock Springer-Verlag, 2001.

\bibitem[Hen83]{co-H-space}
Hans-Werner Henn.
\newblock On almost rational co-h-spaces.
\newblock {\em Proceedings of the american mathematical society}, 87:164--168,
  1983.

\bibitem[Hes07]{hess}
Kathryn Hess.
\newblock Rational homotopy theory: a brief introduction.
\newblock {\em Interactions between Homotopy Theory and Algebra, Contemporary
  Mathematics}, 436:175--202, 2007.

\bibitem[Kad80]{kadeishvili}
Tornike Kadeishvili.
\newblock On the homology theory of fiber spaces.
\newblock {\em English translation: Russian math. surveys}, 35:231--238, 1980.

\bibitem[Kel01]{keller}
Bernard Keller.
\newblock Introduction to a-infinity algebras and modules.
\newblock {\em Homology, homotopy and applications}, 3:1--35, 2001.

\bibitem[LS34]{lscat}
Lazar Lusternik and Lev Schnirelmann.
\newblock Methodes topologiques dans les problemes variationnels.
\newblock {\em Hermann, Paris}, 1934.

\bibitem[Mas58]{Massey}
William~S. Massey.
\newblock Some higher order cohomology operations.
\newblock {\em Symposium internacional de topología algebraica}, page
  145–154, 1958.

\bibitem[Rud99]{Rudyak}
Yuli~B. Rudyak.
\newblock On category weight and its applications.
\newblock {\em Topology 38}, 1:37–55, 1999.

\bibitem[Val14]{AHO}
Bruno Vallette.
\newblock Algebra+homotopy=operad.
\newblock {\em Symplectic, Poisson, and noncommutative geometry}, 62:229–290,
  2014.

\end{thebibliography}

\end{document}